\documentclass{amsart}

\usepackage{amsmath}
\usepackage{amssymb, mathrsfs}
\usepackage{amsthm}
\usepackage[onehalfspacing]{setspace}

\newtheorem{Thm}{Theorem}[section]
\newtheorem{Lem}[Thm]{Lemma}

\theoremstyle{definition}
\newtheorem{Def}[Thm]{Definition}

\theoremstyle{remark}
\newtheorem{Rmk}[Thm]{Remark}

\newcommand{\Rn}{\mathbb{R}^n}
\newcommand{\R}{\mathbb{R}}
\newcommand{\cexp}{exp_x^c}
\newcommand{\cexpz}{exp_{x_0}^c}

\newcommand{\cone}[2]{\mathcal{C}_{#1,#2}}
\newcommand{\icone}[3]{\overline{\mathcal{C}}_{#1, #2}(#3)}
\newcommand{\Int}[1]{\mathrm{Int}(#1)}
\newcommand{\dist}[2]{\mathrm{dist}(#1,#2)}

\newcommand\numberthis{\addtocounter{equation}{1}\tag{\theequation}}

\title[]{Synthetic MTW conditions and their equivalence under mild regularity assumption on the cost function}

\author{Seonghyeon Jeong}

\date{}

\begin{document}

\begin{abstract}
Loeper's condition in \cite{Loe09} and the quantitatively quasi-convex condition (QQconv) from \cite{GK15} are synthetic expressions of the analytic MTW condition from \cite{TW} since they only require $C^2$ differentiability of the cost function $c$. When the cost function $c$ is $C^4$, it is known that the two synthetic MTW conditions are equivalent to the analytic MTW condition. However, when the cost function has regularity weaker than $C^4$, it is not known that if the two synthetic MTW conditions are equivalent. In this paper, we show the equivalence of the synthetic MTW conditions when the cost function has only $C^2$ regularity.
\end{abstract}
\allowdisplaybreaks 

\maketitle

\section{introduction}
In this paper, we will consider the MTW conditions which are necessary and sufficient conditions for the H\"older regularity of solutions to the Monge-Amp\`ere type equations which arise from the optimal transportation problems. The Monge-Amp\`ere type equation is a fully non-linear degenerate elliptic partial differential equation of the form
\begin{equation}
\label{MA}
\det( D^2 u - \mathcal{A}(x, Du) ) = \phi(x, Du),
\end{equation}
where $\mathcal{A} (x,p)= -D^2_{xx}c (x, \cexp ( p) )$ is a matrix valued function defined using the cost function from the optimal transportation problem (see Definition \ref{Def : cexp}) and $\phi : X \times \Rn \to \R$ depends on the cost function and the given data of the optimal transportation problem. The MTW condition is a condition on a (2,2) tensor called the MTW tensor which first appeared in \cite{MTW}. The MTW tensor is defined using the cost function $c$ from the optimal transportation problem :
\begin{displaymath}
MTW = D^2_{pp} \mathcal{A}(x,p).
\end{displaymath} 
In \cite{TW}, the following condition had to be assumed to show the regularity of the solution to the optimal transportation problem : for any $\xi, \eta \in \Rn$ with $\xi \perp \eta$, we have
\begin{equation}
\label{wMTW}
MTW[\xi,\xi, \eta, \eta] = D^2_{pp} \mathcal{A} (x,p) [\xi, \xi, \eta, \eta] \geq 0.
\end{equation}
This condition is called the MTW condition or (A3w). When we have strict inequality in the inequality \eqref{wMTW}, the condition is called (A3s). (A3w) and (A3s) conditions are used to show many regularity results such as H\"older regularity of solutions to \eqref{MA} in \cite{GK15}, \cite{Loe09}, \cite{FKM13} and Sobolev regularity in \cite{PF13} . Moreover, It is proved in \cite{Loe09} that the MTW condition is a necessary and sufficient condition for the H\"older regularity of the solutions to \eqref{MA} when the cost function is $C^4$. 

In \cite{Loe09} and \cite{GK15}, the MTW condition is used to show the H\"older regularity results, but it is used in slightly different forms. It was shown in \cite{Loe09} that the MTW condition is equivalent to \emph{Loeper's condition} (See Definition \ref{Loeper's condition}). On the other hand, in \cite{GK15}, the MTW condition is used in the form of \emph{quantitatively quasi-convexity} (QQconv) (See Definition \ref{QQconvex}). The cost function does not have to be $C^4$ to satisfy Loeper's condition and QQconv. Therefore, we can say that Loeper's condition and QQconv are synthetic MTW conditions and (A3w) is an analytic MTW condition.

In the papers \cite{Loe09} and \cite{GK15}, it is shown that the two conditions, Loeper's condition and quantitatively quasi-convex, are shown to be equivalent to the MTW condition \eqref{wMTW} under the assumption that the cost function is $C^4$. In addition, in the recent work of G. Loeper and N. Trudinger \cite{LT20}, it is shown that a condition called (A3v), which is slightly weaker than Loeper's condition, is still equivalent to the MTW condition when the cost function is $C^4$.  

It is unknown that if the synthetic MTW conditions are equivalent under milder regularity assumption on the cost function (when the cost function is not $C^4$). In this paper, we show that the synthetic MTW conditions are equivalent under $C^2$ regularity assumption on the cost function $c$. The main theorem of this paper is the following :

\begin{Thm}[Main Theorem]
\label{MainThm}
Let $X$ and $Y$ be compact subsets of $\Rn$, and let $c: X \times Y \to \R$ be a measurable function called the cost function. Assume that $c$ is twice differentiable in the sense that there exist the mixed Hessian $D^2_{xy} c$ and $D^2_{yx} c$, and the mixed Hessians are transpose of each other $D^2_{xy} c^T = D^2_{yx} c$. Suppose the conditions (Twisted), (Twisted*), (Non-degenerate), (cDomConv), (cDomConv*) hold. Then if $c$ satisfies Loeper's condition, then $c$ satisfies QQconv.
\end{Thm}

The conditions that we use in the main theorem will be explained in section \ref{Assumptions}. In section \ref{Synthetic MTW}, we define the synthetic MTW type conditions : Loeper's condition and QQconv, and prove some useful lemmas. In section \ref{IntQQconv}, we show that if the cost function satisfies Loeper's condition, then the cost function should satisfy QQconv away from the boundary $\partial Y$. In section \ref{NearBdy}, we show that the cost function satisfying Loeper's condition should satisfy QQconv even near the boundary $\partial Y$, and prove the main theorem. In the final section Section \ref{counter ex}, we construct a counter example when we do not have (Non-degenerate) condition.

We close this section by introducing some results about the MTW conditions. As mentioned, the MTW tensor was first introduced in \cite{MTW}. The synthetic MTW conditions were appear in \cite{Loe09} and \cite{GK15}. In \cite{KM10}, it was shown that the MTW tensor is related to the sectional curvature of the manifold $X \times Y$ with a metric called Kim-McCann metric. Another result that is similar to \cite{KM10} can be found in \cite{GJ20} where some K\"ahler metric on $TM$ was used to show that the MTW tensor is related to the curvature of the  K\"ahler manifold $TM$. Moreover, in \cite{Rankin}, alternative formulation of the MTW condition in terms of monotonicity was established.

\section{Assumptions on the cost function and domains}
\label{Assumptions}
Let $X$ and $Y$ be compact subsets of $\Rn$ with non-empty interiors. Let $c : X \times Y \to \R$ be a function that is $C^2$ up to the boundary in the sense that there exist the mixed Hessian $c_{xy}$ and $c_{yx}$ that satisfy $D^2_{x_i y_j}c = D^2_{y_jx_i}c$. We impose the following conditions on the cost function $c$ : \\

\begin{tabular}{c p{8cm}}
(Twisted) & $-D_x c(x, \cdot ) : Y \to T^*_x X$ is injective for any $x \in X$. \\
(Twisted*) & $-D_y c( \cdot, y) : X \to T^*_y Y$ is injective for any $y \in Y$. \\
(Non-degenerate) & The mixed Hessian $D^2_{xy}c(x,y)$ is invertible for any $(x,y) \in X \times Y$.
\end{tabular}\\

These conditions are commonly used in many optimal transportation problems. The conditions (Twisted) and (Twisted*) imply the existence of inverse functions of $-D_x c(x, \cdot)$ and $-D_y c(\cdot, y)$. We name these inverse functions $c$-exponential map and $c^*$-exponential map respectively.

\begin{Def}
\label{Def : cexp}
Let $x \in X$ and $y \in Y$. We denote the image of $Y$, $X$ under the map $-D_x c( x, \cdot )$ and $-D_y c(\cdot, y)$ by $Y^*_{x}$ and $X^*_{y}$ respectively :
\begin{displaymath}
Y^*_{x} = -D_x c(x, Y), \ X^*_y = -D_y c(X,y).
\end{displaymath}
We define \emph{$c$-exponential map} $\cexp : Y^*_x \to Y$ and \emph{$c^*$-exponential map} $exp^{c^*}_y : X^*_y \to X$ by the inverse function of $-D_x c( x , \cdot )$ and $ -D_y c(\cdot, y)
$ respectively :
\begin{displaymath}
-D_x c(x, \cexp(p) ) = p, \ -D_y c( exp^{c^*}_y (q), y ) = q.
\end{displaymath}
\end{Def}

Since the cost function has a continuous mixed Hessian $D^2_{xy}c$ on the compact set $X \times Y$, it has a modulus of continuity $\omega$
\begin{equation*}
    \| D^2_{xy}c(x_1,y_1) - D^2_{xy}c(x_0, y_0) \| \leq \omega(| (x_1, y_1) - (x_0, y_0)|).
\end{equation*}
where $\| D^2_{xy} c \|$ is the Frobenius norm of $D^2_{xy} c$. Note that (Non-degenerate) condition with compactness of domain $X \times Y$ implies existence of some positive constant $\alpha$ such that
\begin{equation*}
\frac{1}{\alpha} \leq \| D^2_{xy} c \| \leq \alpha, \quad \frac{1}{\alpha} \leq \|(D^2_{xy}c)^{-1}\| \leq \alpha
\end{equation*}
Then the inverse of $D^2_{xy}c$, which exists due to the (Non-degenerate) condition, has a modulus of continuity $\alpha^2 \omega$
\begin{equation*}
    \begin{aligned}
        &\left\| \left( D^2_{xy}c(x_1, y_1) \right)^{-1}  - \left( D^2_{xy}c(x_0, y_1)\right)^{-1} \right\| \\ 
        =& \left\| \left( D^2_{xy}c(x_1, y_1) \right)^{-1} \right\| \left\| \left( D^2_{xy}c(x_0, y_1)\right)^{-1}\right\|  \left\|  D^2_{xy}c(x_1, y_1) -  D^2_{xy}c(x_0, y_1) \right\| \\
        \leq & \alpha^2 \omega ((x_1, y_1) - (x_0, y_0))
    \end{aligned}
\end{equation*}

Moreover, (Non-degenerate) condition implies that the functions $-D_x c(x, \cdot)$ and $-D_y c( \cdot , y)$ are bi-Lipschitz :
\begin{displaymath}
\frac{1}{\lambda} |y_1 - y_0| \leq |-D_x c(x, y_1) + D_x c(x, y_0) | \leq \lambda |y_1 - y_0|,
\end{displaymath}
\begin{displaymath}
\frac{1}{\lambda} |x_1 - x_0 | \leq |- D_y c(x_1, y ) + D_y c(x_0, y) | \leq \lambda |x_1 - x_0|.
\end{displaymath}
 
In addition to the conditions on the cost function $c$, we need to impose some geometric conditions on the domain $X$ and $Y$ : \\

\begin{tabular}{c p{8cm}}
($c$DomConv) & $Y^*_x$ is convex for any $x \in X$. \\
($c$DomConv*) & $X^*_y$ is convex for any $y \in Y$. 
\end{tabular}\\

In the rest of the paper, we always assume the conditions explained in this section hold.

\section{The synthetic MTW conditions}
\label{Synthetic MTW}
\begin{Def}
The MTW tensor of the cost function $c : X \times Y \to \R$ is a (2,2) tensor defined by
\begin{displaymath}
MTW = D^2_{pp} A(x,p) 
\end{displaymath}
where $A(x,p) = - D^2_{xx} c(x, \cexp(p))$. 
The cost function $c$ said to satisfy (\emph{A3w}) condition if
\begin{equation}
\label{A3w}
MTW[\xi, \xi, \eta, \eta] \geq 0
\end{equation}
holds for any $\xi \perp \eta$. If the cost function satisfy the inequality \eqref{A3w} with strict inequality, then the cost function is said to satisfy (\emph{A3s}).
\end{Def}

The following two conditions in next two definitions are equivalent to (A3w) when the cost function is $C^4$. See \cite{Loe09} and \cite{GK15} for the proofs. However, we do not need $C^4$ regularity for the definitions.

\begin{Def}
\label{Loeper's condition}
Let $x_0 \in X$ and $v_1, v_0 \in Y^*_{x_0}$, and let $v_t = (1-t)v_0 + tv_1, \forall t \in [0,1]$ and $f_t(x) = -c(x, \cexpz(v_t)) +c(x_0, \cexpz(v_t))$. The cost function $c$ is said to satisfy \emph{Loeper's condition} if it satisfies 
\begin{equation}
\label{Loeper's ineq}
f_t(x) \leq \max\{ f_1(x) , f_0(x) \}
\end{equation} 
for any $x, x_0 \in X$ and $v_1, v_0 \in Y^*_{x_0}$.
\end{Def}

\begin{Def}
\label{QQconvex}
Let $f_t$ be as in Definition \ref{Loeper's condition}. The cost function is said to be \emph{quantitatively quasi-convex}(QQconv) if there exists $M \geq 1$ such that 
\begin{equation}
\label{QQconvex ineq}
f_t(x) - f_0(x) \leq Mt (f_1(x) - f_0(x))_+
\end{equation}
holds for any $x, x_0 \in X$ and $v_1, v_0 \in Y^*_{x_0}$.
\end{Def}

\begin{Rmk}
It is easy to see that QQconv implies Loeper's condition. In fact, if $f_0(x) \geq f_1(x)$ then 
\begin{displaymath}
f_t(x) \leq Mt(f_1(x) - f_0(x))_+ + f_0(x) = f_0(x).
\end{displaymath}
If $f_0(x) < f_1(x)$, then set $v'_0 = v_1$ and $v'_1 = v_0$ and use \eqref{QQconvex ineq} like above.
\end{Rmk}

Hereafter, we use notation $v_t = (1-t)v_0 + t v_1$ and 
\begin{equation}
\label{Def : F}
F(v) = -c(x_1, \cexpz(v)) + c(x_0, \cexpz(v))
\end{equation}
Then \eqref{Loeper's ineq} becomes 
\begin{equation}
\label{Qconv F}
F(v_t) \leq \max\{ F(v_1), F(v_0) \}
\end{equation}
which is the quasi-convexity of the function $F : Y^*_{x_0} \to \R$.  \eqref{QQconvex ineq} can be expressed using $F$ as well :
\begin{equation}
\label{QQconvex ineq 2}
F(v_t) -F(v_0) \leq Mt( F(v_1) - F(v_0))_+.
\end{equation}

\begin{Rmk}
\label{positive case only}
Suppose the cost function satisfies Loeper's condition. If $F(v_1) \leq F(v_0)$, then Loeper's condition implies $F(v_t) \leq F(v_0)$ so that \eqref{QQconvex ineq 2} holds. Therefore, when we show the implication from Loeper's condition to QQconv, we only need to check the case when $F(v_1) > F(v_0)$.
\end{Rmk}

\begin{Rmk}
\label{t must small}
Suppose the cost function satisfies Loeper's condition. By Remark \ref{positive case only}, consider $F(v_1) > F(v_0)$. If we have opposite inequality of \eqref{QQconvex ineq 2} i.e.
\begin{displaymath}
F(v_t) - F(v_0) > Mt(F(v_1) - F(v_0)),
\end{displaymath}
then we get $t < \frac{1}{M}$. In fact, Loeper's condition implies
\begin{displaymath}
Mt(F(v_1) - F(v_0)) < F(v_t) - F(v_0) \leq F(v_1) - F(v_0).
\end{displaymath}
This shows $t < \frac{1}{M}$.
\end{Rmk}

The next lemma shows that, to prove 'Loeper's condition $\Rightarrow$ QQconv', we only need to verify \eqref{QQconvex ineq 2} locally. 

\begin{Lem}
\label{local QQconv implies global}
Suppose the cost function $c$ satisfies Loeper's condition. If there exists $r>0$ such that \eqref{QQconvex ineq 2} holds for any $x_0, x_1 \in X$ and  $v_0, v_1 \in Y^*_{x_0}$ with $|v_1 - v_0| \leq r$, then the cost function is QQconv.
\end{Lem}
\begin{proof}
Suppose we have such $r$, and let $v_0,v_1 \in Y^*_{x_0}$ be such that $|v_1 - v_0| > r$. Let $M' > \frac{\lambda }{r}\mathrm{diam}(Y) \geq \frac{1}{r} \mathrm{diam}(Y^*_{x_0})$. If $v_1$ satisfies \eqref{QQconvex ineq 2} for any $t \in [0,1]$ with the constant $M'$ instead of $M$, then we are done. Therefore we assume that 
\begin{equation}
\label{local QQconv implies global : counter ex}
F(v_t) - F(v_0) > M't(F(v_1) - F(v_0)) > 0.
\end{equation}
By Remark \ref{t must small}, we have $t < r / \mathrm{diam}(Y^*_{x_0})$ and
\begin{displaymath}
|v_t - v_0| = t|v_1 - v_0| \leq t \mathrm{diam}(Y^*_{x_0}) < r.
\end{displaymath}
Note that $s \mapsto \frac{1}{s} (F(v_s) - F(v_0))$ is continuous on $\left[ t,\frac{1}{M'} \right]$ and the value of this map is bigger than $M'(F(v_1) - F(v_0) )$ at $s=t$ by \eqref{local QQconv implies global : counter ex}. In addition, Loeper's condition implies $F(v_{1/M'}) \leq F(v_1)$ so that when $s=\frac{1}{M'}$ we have $M'(F(v_{1/M'}) - F(v_0)) \leq M'(F(v_1) - F(v_0))$. Therefore, by the intermediate value theorem, we can choose $t' \in \left( t, \frac{1}{M'} \right]$ such that
\begin{equation}
\label{local QQconv implies global : t'}
\frac{1}{t'}(F(v_{t'}) - F(v_0)) = M'(F(v_1) - F(v_0) ). 
\end{equation}
Note that 
\begin{displaymath}
|v_0 - v_{t'}| = t' |v_1 - v_0| \leq \frac{1}{M'} |v_1 - v_0| \leq r.
\end{displaymath} \\
Combining the assumption of this lemma (\eqref{QQconvex ineq 2} holds if $|v_1-v_0|<r$), \eqref{local QQconv implies global : counter ex} and \eqref{local QQconv implies global : t'}, we obtain
\begin{displaymath}
F(v_t) - F(v_0) \leq M \left( \frac{t}{t'} \right) (F(v_{t'}) - F(v_0)) = MM' t(F(v_1) - F(v_0)).
\end{displaymath}
Note that the constant $MM'$ is uniform over $x_0, x_1$ and $v_0,v_1$. Therefore the cost function is QQconv with constant $MM'$.
\end{proof}

For fixed $x_1, x_0 \in X$ and for $v_0 \in Y^*_{x_0}$, we use notation
\begin{equation}
\label{notation : half ball}
B_r^+(v_0) = \left\{ v \in B_r (v_0) | \langle v-v_0 , \nabla F(v_0) \rangle \geq 0 \right\}.
\end{equation}
For the level sets and sublevel sets of $F$, we use notations
\begin{equation}
\label{notation : level set}
L_{v_0} = \{ v | F(v) = F(v_0)\}, \ SL_{v_0} = \{ v | F(v) \leq F(v_0) \}.
\end{equation}

\begin{Rmk}
    \label{Rmk : increasing in half ball}
    Since $F$ is quasi-convex when Loeper's condition holds, $SL_{v_0}$ is convex. In addition, since $\nabla F(v_0)$ is the outward normal vector of $SL_{v_0}$, we have 
    \begin{displaymath}
        SL_{v_0} \subset \{ v | \langle v - v_0, \nabla F(v_0) \rangle \leq 0 \}.
    \end{displaymath}
    Therefore, the convexity of $SL_{v_0}$ implies 
    \begin{equation}
    \label{eqn: half ball in suplv}
        F(v) \geq F(v_0), \quad \forall v \in B^+_r(v_0).
    \end{equation}
    Moreover, for any $0 \leq s \leq t \leq 1$, we also have
    \begin{equation*}
        F(v_s) \leq F(v_t).
    \end{equation*}
    Indeed, we already have $v_0, v_t \in SL_{v_t}$. Then the convexity of $SL_{v_t}$ yields $v_s \in SL_{v_t}$, which concludes the above inequality.
\end{Rmk}

In the following lemma, we observe that it is enough to show 'Loeper's condition $\Rightarrow$ QQconv' when $v_1 \in B^+_r(v_0)$ for some $r>0$.

\begin{Lem}
\label{Lem : QQconv on half ball to global}
Suppose the cost function $c$ satisfies Loeper's condition. If there exists $r>0$ such that \eqref{QQconvex ineq 2} holds for any $x_0,x_1 \in X$, $v_0 \in Y^*_{x_0}$ and $v_1 \in B^+_r(v_0) \cap Y^*_{x_0}$, then the cost function is QQconv.
\end{Lem}
\begin{proof}
Suppose we have such $r$.  Suppose $v_0 \in Y^*_{x_0}$, then by the proof of Lemma \ref{local QQconv implies global}, for any $v_1 \in Y^*_{x_0} \cap \{ v | \langle v-v_0, \nabla F(v_0) \rangle \geq 0 \}$, we obtain the inequality \eqref{QQconvex ineq 2} for some $M \geq 1$. Hence we need to show the inequality \eqref{QQconvex ineq 2} when $v_1 \in Y^*_{x_0} \cap \{ v | \langle v-v_0 , \nabla F(v_0) \rangle <0 \}$. Moreover, by Remark \ref{positive case only}, we can assume $F(v_1) \geq F(v_0)$. \\
Let $v_1 \in Y^*_{x_0} \cap \{ v | \langle v-v_0 , \nabla F(v_0) \rangle <0 \}$ and $F(v_1) \geq F(v_0)$. Then $\langle v_1 - v_0, \nabla F(v_0) \rangle < 0$ implies that $F(v_{\epsilon}) < F(v_0)$ for some $\epsilon >0$ small. Hence, by the intermediate value theorem, there exists $t' \in (\epsilon,1) $ such that $F(v_{t'}) = F(v_0)$. Moreover, we have $ \langle v_1 - v_{t'}, \nabla F(v_{t'}) \rangle > 0$ since $v_{t'}$ is in the intersection of $L_{v_0}$, the boundary of the convex set $SL_{v_0}$, and a ray that starts from $v_{\epsilon}$ which is an interior point of $SL_{v_0}$. Also, note that $|v_1 - v_{t'}| \leq |v_1-v_0| \leq r$. Then the assumption of this lemma (QQconv holds if $v_1 \in B^+_r(v_0) \cap Y^*_{x_0}$) yields
\begin{align*}
F(v_s) - F(v_0) & = F(v_{s})- F(v_{t'}) \\
& \leq M \frac{s-t'}{1-t'}( F(v_{1}) - F(v_{t'}) ) \leq M s ( F(v_1) - F(v_0))
\end{align*}
for $s \in [t', 1]$. If $s \in [0,t')$, then $F(v_{s}) \leq F(v_0)$ and by Remark \ref{positive case only}, we get the inequality \eqref{QQconvex ineq 2}. Therefore the inequality \eqref{QQconvex ineq 2} holds for $v_1 \in Y^*_{x_0} \cap \{ v | \langle v-v_0 , \nabla F(v_0) \rangle <0 \} \cap B_r(v_0)$ too, and Lemma \ref{local QQconv implies global} implies that the cost function $c$ is QQconv.
\end{proof}

We close this section by introducing an estimate for the modulus of continuity of $\nabla F$. 

\begin{Lem}
\label{Lem : Lip grad F}
for any $x_1, x_0 \in X$ and $v_1, v_0 \in Y^*_{x_0}$, we have
\begin{equation}
\label{Lem : Lip grad F / Lip ineq}
| \nabla F(v_1) - \nabla F(v_0) | \leq C |x_1 - x_0| \omega(\lambda|v_1 - v_0 |)
\end{equation}
for some constant $C$ that depends on $\alpha$ and $\lambda$.
\end{Lem}
\begin{proof}
By the definition of $F$ \eqref{Def : F}, we can compute
\begin{equation}
\label{grad F}
\nabla F(v) = [-D^2_{yx} c(x_0,y)]^{-1} \left( -D_y c(x_1, y) + D_y c(x_0, y) \right)
\end{equation}
where $y = \cexpz(v)$. Note that we used $D^2_{x_i y_j} c = D^2_{y_j x_i}c$. Hence we have
\begin{equation}
\label{Lem : Lip grad F / difference of grad}
    \begin{aligned}
        & \nabla F(v_1) - \nabla F(v_0) \\
        = &[-D^2_{yx} c(x_0,y_1)]^{-1} \left( -D_y c(x_1, y_1) + D_y c(x_0, y_1)\right) \\
        & - [-D^2_{yx} c(x_0,y_0)]^{-1} \left( -D_y c(x_1, y_0) + D_y c(x_0, y_0)\right) 
    \end{aligned}
\end{equation}
where $y_i = \cexpz(v_i)$. Let $L1$ and $L2$ be the second and third line in \eqref{Lem : Lip grad F / difference of grad} so that $\nabla F(v_1) - \nabla F(v_0) = L1 + L2$. Let
\begin{align*}
& L1' = L1 - [-D^2_{yx} c(x_0,y_0)]^{-1}\left( -D_y c(x_1, y_1) + D_y c(x_0, y_1)\right) \\
& L2' = L2 + [-D^2_{yx} c(x_0,y_0)]^{-1}\left( -D_y c(x_1, y_1) + D_y c(x_0, y_1)\right),
\end{align*}
then $\nabla F(v_1) - \nabla F(v_0) = L1' + L2'$. We use the modulus of continuity of the mixed Hessian $D^2_{xy}c$ on $L1'$ to obtain
\begin{equation}
\label{Lem : Lip grad F / L1'}
    \begin{aligned}
        |L1'| & = \left\| [-D^2_{yx} c(x_0,y_1)]^{-1} - [-D^2_{yx} c(x_0,y_0)]^{-1} \right\| \left| -D_y c(x_1, y_1) + D_y c(x_0, y_1)\right| \\
        & \leq \alpha^2 \omega (|y_1 - y_0|) \times \lambda |x_1 - x_0| \\
        & \leq \lambda \alpha^2 \omega(\lambda | v_1 - v_0 |) | x_1 - x_0 |.
    \end{aligned}
\end{equation}
To get an estimate for $L2'$, we calculate
\begin{equation*}
    \begin{aligned}
        &|L2'| \\
& = \left\| [-D^2_{yx} c(x_0,y_0)]^{-1} \right\| \left| -D_y c(x_1, y_1) + D_y c(x_0, y_1) + D_y c(x_1, y_0) - D_y c(x_0, y_0) \right| \\
& = \left\| [-D^2_{yx} c(x_0,y_0)]^{-1} \right\| \\
& \times \left| \int_0^1 [-D^2_{yx}c(x_s,y_0)]^{-1} [-D^2_{yx} c(x_s, y_1) ] (q_1 - q_0) ds - (q_1 - q_0) \right|  \\
& = \left\| [-D^2_{yx} c(x_0,y_0)]^{-1} \right\|  \\
&  \times \left| \int_0^1  [-D^2_{yx}c(x_s,y_0)]^{-1} \left( [-D^2_{yx} c(x_s, y_1) ] - [-D^2_{yx}c(x_s,y_0)] \right) ds (q_1 - q_0) \right|
    \end{aligned}
\end{equation*}
where $q_i = -D_y c(x_i, y_0)$ and $x_s = exp^{c*}_{y_0}( (1-s)q_0 + s q_1 )$. We use the Lipschitzness of $D_y c$ and the modulus of continuity of $D^2_{yx}c$ to obtain
\begin{equation}
\label{Lem : Lip grad F / L2'}
|L2'| \leq \lambda \alpha^2 \omega(\lambda|v_1 - v_0|)|x_1 - x_0|.
\end{equation}
We combine \eqref{Lem : Lip grad F / L1'} and \eqref{Lem : Lip grad F / L2'} to obtain the desired result with $C = 2\lambda \alpha^2$.
\end{proof}

\begin{Rmk}
\label{Remark : grad F comparable to x1-x0}
From the equation \eqref{grad F}, we can get
\begin{displaymath}
|\nabla F(v)| \sim |x_1 - x_0|.
\end{displaymath}
In particular, there exists a constant $C_1$ that depends on $\alpha$ and $\lambda$ such that
\begin{equation}
\label{Remark : grad F comparable to x1-x0 / lower bound}
| \nabla F(v) | \geq C_1 |x_1 - x_0|.
\end{equation}
\end{Rmk}

\section{Interior QQconv}
\label{IntQQconv}
For $v_0 \in Y^*_{x_0}$, we denote the cone with vertex $v_1$ and axis $\nabla F(v_0)$ by
\begin{equation}
\label{notation : cone}
\cone{k}{v_0} = \left\{ v | \langle v-v_0 , \nabla F(v_0) \rangle \geq \frac{1}{k}|v-v_0||\nabla F(v_0) | \right\}.
\end{equation}

\begin{Lem}
\label{Lem : QQconv on sec cone}
Let $c$ be the cost function that satisfies Loeper's condition and let $k \geq 1$. Then there exists $r_k >0$ such that if $v_0 \in Y^*_{x_0}$ and $v_1 \in \cone{k}{v_0} \cap B_{r_k}(v_0) \cap Y_{x_0}^*$, then 
\begin{displaymath}
F(v_t) - F(v_0) \leq 5t (F(v_1) - F(v_0)).
\end{displaymath}
\end{Lem}
\begin{proof}
Note first that \eqref{eqn: half ball in suplv} implies 
\begin{equation*}
F(v_1)-F(v_0) \geq 0 , \quad \forall v_1 \in \cone{k}{v_0} \cap B_{r_k}(v_0) \cap Y_{x_0}^*
\end{equation*} 
Let $r_k>0$ be such that $ \omega(\lambda r_k ) \leq  \frac{C_1}{2Ck}$, where $C_1$ is from Remark \ref{Remark : grad F comparable to x1-x0} and $C$ is from Lemma \ref{Lem : Lip grad F}. Note that
\begin{equation}
\label{Lem : QQconv on sec cone / cutting}
\cone{k}{v_0} \cap B_{r_k}(v_0) \subset \overline{ \bigcup_{i=0}^{\infty} \left[ \left( \cone{\frac{k}{2^i}}{v_0} \cap B_{r_k}(v_0) \right) \setminus \cone{\frac{k}{2^{i+1}}}{v_0} \right]}.
\end{equation}
Let $k_i = \frac{k}{2^i}$. Note that $r_{k_i}$ increases as $i$ increases. Then, for $v \in B_{r_k} \cap Y_{x_0}^* \subset B_{r_{k_i}}$, we have
\begin{equation}
\label{grad est in small ball}
    \begin{aligned}
        | \nabla F(v) - \nabla F(v_0) | &\leq C |x_1 - x_0 | \omega(\lambda|v - v_0|) \leq C |x_1 - x_0 | \omega(\lambda r_k)  \\
        &\leq \frac{C_1}{2k_i} |x_1 - x_0| \leq \frac{1}{2k_i} | \nabla F(v_0)|,
    \end{aligned}
\end{equation}
where we have used \eqref{Remark : grad F comparable to x1-x0 / lower bound} in the last inequality. Let $v_1 \in \left( \cone{k_i}{v_0} \cap B_{r_k}(v_0) \cap Y_{x_0}^*\right) \setminus \cone{\frac{k_i}{2}}{v_0}$. Then we have
\begin{align*}
\langle \nabla F(v_t), v_1 - v_0 \rangle & = \langle \nabla F(v_t) - \nabla F(v_0) , v_1 - v_0 \rangle + \langle \nabla F(v_0) , v_1 - v_0 \rangle \\
& \geq - \frac{1}{2k_i} |\nabla F(v_0)||v_1 - v_0| + \frac{1}{k_i} | \nabla F(v_0) ||v_1 - v_0|\\
&= \frac{1}{2k_i} | \nabla F(v_0) | |v_1 - v_0 |
\end{align*}
where we have used \eqref{grad est in small ball} and $v_1 \in \cone{k_i}{v_0}$ for the above inequality. Therefore, 
\begin{equation}
\label{Lem : QQconv on sec cone / lower est}
    \begin{aligned}
        F(v_1) - F(v_0) & = \int_0^{1} \langle \nabla F (tv_1 +(1-t)v_0), v_1 - v_0 \rangle ds \\
        & \geq \frac{1}{2k_i} |v_1 - v_0||\nabla F(v_0) |.
    \end{aligned}
\end{equation}
Moreover, 
\begin{equation*}
    \begin{aligned}
        \langle \nabla F(v_s),v_1 - v_0 \rangle & = \langle \nabla F(v_s) - \nabla F(v_0), v_1 - v_0 \rangle + \langle \nabla F(v_0) , v_1 - v_0 \rangle \\
        & \leq \frac{1}{2k_i} | \nabla F(v_0) ||v_1 - v_0| + \frac{2}{k_i} | \nabla F(v_0) ||v_1 - v_0|\\
        &= \frac{5}{2k_i} | \nabla F(v_0) ||v_1 - v_0|,
    \end{aligned}
\end{equation*}
where we have used $v_1 \notin \cone{\frac{k_i}{2}}{v_0} $ in the above inequality. Therefore,
\begin{equation}
\label{Lem : QQconv on sec cone / upper est}
    \begin{aligned}
        F(v_t) - F(v_0) & = \int_0^{t} \langle \nabla F(v_s) , v_1 - v_0 \rangle ds \\
& \leq t|v_1 - v_0| \times \frac{5}{2k_i} | \nabla F(v_0) |.
    \end{aligned}
\end{equation}
Combining \eqref{Lem : QQconv on sec cone / lower est} and \eqref{Lem : QQconv on sec cone / upper est}, we get the desired inequality for $v_1 \in \left( \cone{k_i}{v_0} \cap B_{r_{k_i}  }(v_0) \cap Y_{x_0}^*\right) \setminus \cone{\frac{k_i}{2}}{v_0}$. \\
Now, for any $v_1 \in \cone{k}{v_0} \cap B_{r_k} \cap Y_{x_0}^*$, \eqref{Lem : QQconv on sec cone / cutting} implies that there exists a sequence $v^j_1$ such that $v^j_1 \in \left( \cone{\frac{k}{2^i}}{v_0} \cap B_{r_k}(v_0) \right) \setminus \cone{\frac{k}{2^{i+1}}}{v_0}$ for some $i$ and converges to $v_1$ as $j \to \infty$. Then we have
\begin{equation*}
F(v^j_t) - F(v_0) \leq 5t (F(v^j_1)-F(v_0)),
\end{equation*}
where $v^j_t = tv^j_1 + (1-t) v_0$. Taking $j \to \infty$, in the above inequality, we obtain the desired inequality for any $v_1 \in \cone{k}{v_0} \cap B_{r_k} \cap Y_{x_0}^*$.
\end{proof}

\begin{Rmk}
\label{Remark : comparable cone}
From the equation \eqref{grad est in small ball}, we know that if $v \in B_{r_k}(v_0)$, then
\begin{equation}
\label{grad in small ball}
\nabla F (v) \in B_{\frac{\rho_0}{2k}}(\nabla F(v_0) )
\end{equation}
where $\rho_0 = | \nabla F (v_0) |$. Let $\nabla F(v) =\nabla F(v_0) + u$ where $ |u| \leq \frac{\rho_0}{2k}$, and consider $\cone{k'}{v_0}$. For any $\nu$ such that $\nu + v_0 \in \cone{k'}{v_0}$ with $| \nu |=1$, we have
\begin{equation}
\label{lower comparability}
\langle \nu , \nabla F(v) \rangle = \langle \nu  , \nabla F(v_0) \rangle + \langle \nu  , u \rangle \geq \frac{1}{k'} | \nabla F(v_0) | - \frac{1}{2k} | \nabla F(v_0) | .
\end{equation}
Therefore, once we fix $k$ and $k' < 2k$, we get $\langle \nu , \nabla F(v) \rangle \sim | \nabla F(v_0) |$.
\end{Rmk}

\begin{Rmk}
\label{Remark : icone}
Suppose $B_{r_k}(v_0) \subset Y^*_{x_0}$. Since $F$ is quasi-convex, the sublevel set $SL_{v_0}$ is convex and $\nabla F(v_0)$ is the outward normal vector. Moreover, for any $v' \in L_{v_0}\cap B_{r_k}(v_0)$, we know \eqref{grad in small ball} holds. Then
\begin{align*}
SL_{v_0} \cap B_{r_k}(v_0) & = \left( \bigcap_{v'}  \{ v | \langle v-v' , \nabla F(v') \rangle  \leq 0 \} \right) \cap B_{r_k}(v_0) \\
& \supset \left( \bigcap_{v'} \{ v | \langle v - v_0 , \nabla F(v') \rangle \leq 0 \} \right) \cap B_{r_k}(v_0) \\
& \supset \left\{ v | \langle v - v_0, u \rangle \leq 0, \forall u \in B_{\frac{\rho_0}{2k}}(\nabla F(v_0)) \right\} \cap B_{r_k}(v_0).
\end{align*}
where $\displaystyle \bigcap_{v'}$ means the intersection over $v' \in L_{v_0} \cap B_{r_k}(v_0)$. In particular, this shows that $v_0 - \frac{r_k}{2 \rho_0}\nabla F(v_0) \in SL_{v_0} \cap B_{r_k}(v_0)$. Now consider the cone
\begin{equation}
\label{notation : icone}
\icone{k'}{v_0}{v_1} = \left\{ v | \langle v-v_1 , \nabla F(v_0) \rangle \leq - \frac{1}{k'} |v-v_1|| \nabla F(v_0) | \right\}.
\end{equation}
Let $k' \geq 3$ then for any $v_1 \in B^+_{r_k}(v_0)$, 
\begin{align*}
\langle v_0 - \frac{r_k}{2\rho_0} \nabla F(v_0) - v_1, \nabla F(v_0) \rangle & = -\langle v_1 - v_0 , \nabla F(v_0) \rangle - \frac{r_k}{2 \rho_0} \langle \nabla F(v_0), \nabla F(v_0) \rangle \\
& \leq - \frac{r_k}{2} | \nabla F(v_0) |  \\
& = - \frac{1}{3}\times \frac{3r_k}{2} \times | \nabla F(v_0) |\\
& \leq - \frac{1}{k'} \times |v_0 - \frac{r_k}{2\rho_0} \nabla F(v_0) - v_1| \times | \nabla F(v_0) |.
\end{align*}
Therefore $\icone{k}{v_0}{v_1}$ contains $v_0 - \frac{r_k}{2 \rho_0}\nabla F(v_0)$ for any $v_1 \in B^+_{r_k}(v_0)$, which implies that $\icone{k'}{v_0}{v_1} \cap L_{v_0} \cap B_{r_k}(v_0) \neq \emptyset$.
\end{Rmk}

\begin{Lem}
\label{Lem : local QQconv}
Fix $3\leq k' < k$. Suppose $B_{r_k}(v_0) \subset Y^*_{x_0}$. Then, for any $v_1 \in B^+_{r_k}(v_0)$, we have
\begin{displaymath}
F(v_t) - F(v_0) \leq M_{k'}t( F(v_1) - F(v_0) )
\end{displaymath}
for some $k'>0$ and $M_{k'}$ a constant depending on $k'$.
\end{Lem}
\begin{proof}
Note that by Lemma \ref{Lem : QQconv on sec cone}, we have the result when $v_1 \in \cone{k}{v_0}$. Let $v_1 \in B^+_{r_k}(v_0) \setminus \cone{k}{v_0}$ and consider the cone $\icone{k'}{v_0}{v_1}$(recall \eqref{notation : icone}). By Remark \ref{Remark : icone}, we have that $\icone{k'}{v_0}{v_1} \cap L_{v_0} \cap B_{r_k}(v_0) \neq \emptyset$. Let $ u_1 \in \icone{k'}{v_0}{v_1} \cap L_{v_0} \cap B_{r_k}(v_0)$ and let $\nu = (v_1 - u_1)/|v_1 - u_1|$. We claim that $v_t - s_t \nu \in L_{v_0}$ for some $s_t $. In fact, taking $s = t|v_1 - u_1|$, we get
\begin{displaymath}
v_t - s \nu = tu_1 + (1-t)v_0 \in SL_{v_0}.
\end{displaymath}
On the other hand, $v_t \in B^+_{r_k}(v_0)$ and \eqref{eqn: half ball in suplv} yield $F(v_t) \geq F(v_0)$. Therefore, by the intermediate value theorem, $v_t - s_t \nu \in L_{v_0}$ for some $s_t \in [0,  t|v_1 - u_1|]$. Now, up to an isometry, we can set $\nu = - e_n$, $v_0 = 0$, and $v_1 = ae_1 + be_n$ for some $a,b \in \R$, $a>0$. Then, noting that $\nu$ is not parallel to $v_0 - v_1$ by our choice of $v_1$, we can view the set $\{ v_t - s_t \nu | t \in [0,1] \}$ as a graph of a function $g$ on $[0,ae_1]$. Moreover, the convexity of $SL_{v_0}$ implies that the epigraph of $g$ is convex i.e. $g$ is a convex function. Note that $s_t = g(at) - bt$ so that $s_t$ is also a convex function of $t$ on $[0,1]$. Then, using $s_0 = 0$, we obtain
\begin{equation}
\label{Lem : local QQconv / convex parametrization}
|v_t - u_t| = s_t \leq t s_1 = t | v_1 - u_1 |
\end{equation} 
where $u_t = v_t - s_t \nu \in L_{v_0}$. Moreover, recalling Remark \ref{Remark : comparable cone}, we obtain the following:
\begin{align*}
\label{Lem : local QQconv / est left}
F(v_t) - F(v_0) & = F(v_t) - F(u_t) \\
& = \int_0^{s_t} \langle \nabla F ( u_t + s \nu) , \nu \rangle ds  \leq s_t \frac{2k+1}{2k} | \nabla F(v_0) |, \numberthis
\end{align*}
and
\begin{align*}
\label{Lem : local QQconv / est right}
F(v_1) - F(v_0) & = F(v_1) - F(u_1) \\
& = \int_0^{s_1} \langle \nabla F(u_1 + s \nu), \nu \rangle ds  \geq s_1 \left( \frac{1}{k'} - \frac{1}{2k} \right) | \nabla F(v_0) |. \numberthis
\end{align*}
We combine \eqref{Lem : local QQconv / convex parametrization}, \eqref{Lem : local QQconv / est left}, and \eqref{Lem : local QQconv / est right}, and we obtain
\begin{align*}
F(v_t) - F(v_0) & \leq 2 s_t | \nabla F(v_0) | \\
& \leq 2t s_1 | \nabla F(v_0) | \leq 2k' (F(v_1) - F(v_0))
\end{align*}
which gives the desired inequality with $M_{k'} = 2k'$.
\end{proof}

\begin{Rmk}
Lemma \ref{Lem : local QQconv} and the proof of Lemma \ref{local QQconv implies global} shows that if we have a set $Y' \subset Y$ such that $ \mathrm{dist}(Y', \partial Y) = d$, then the cost function is QQconv on $Y'$ with the constant $M \sim \frac{1}{d}$.
\end{Rmk}

\section{Near the boundary}
\label{NearBdy}
One problem of the argument that we used in the proof of Lemma \ref{Lem : local QQconv} is the following : if we fix $r_k$, the radius of the ball centered at $v_0$, then as we pick $v_0$ close to the boundary $\partial Y^*_{x_0}$, the ball $B_{r_k}(v_0)$ may intersect with the boundary $\partial Y^*_{x_0}$. Then the cone $\icone{k'}{v_0}{v_1}$ may not intersect with the level set $L_{v_0}$ and we cannot get the function $s_t$ in the proof of Lemma \ref{Lem : local QQconv}. 

To bypass this problem, we divide the segment $[v_0, v_1]$ into $m$ pieces, where we can use the idea of Lemma \ref{Lem : local QQconv} in each piece. More precisely, we use the convex function $\sigma$ on the piece that contains $v_t$ and obtain an estimate for $F(v_t) - F(v_{j/m})$ for some $j$. Then we combine the estimate with other estimates for $F(v_{i/m}) - F(v_{(i-1)/m})$ to obtain the $t$ factor in \eqref{QQconvex ineq 2}. While combining the estimates, we see that the number $m$, which depends on $v_0$ and $v_1$, disappears, and we obtain the constant $M$ that is uniform over $x_0, x_1 \in X$ and $v_0, v_1 \in Y^*_{x_0}$.

\begin{Lem}
    \label{Lem : subdivision}
    Let $v_0, v_1 \in \Int{Y_{x_0}^*}$ and suppose $v_1 \in \Int{B^+_{r_k/2}(v_0)}$. Then there exists a positive number $m$ and a positive real number $\epsilon$ that depends on $|v_1 - v_0|$ and $\dist{[v_0, v_1]}{\partial Y^*_{x_0}}$ such that
    \begin{equation}
    \label{eqn: m epsilon}
        \frac{\epsilon}{2} \leq \frac{|v_1 - v_0|}{m} \leq \epsilon
    \end{equation}
    and, for any $0 \leq i \leq m-1$ and for any $t \in [\frac{i}{m}, \frac{i+1}{m}]$, there exists $u_t$ such that 
    \begin{equation*}
        u_t \in \{v_t-s\nabla F(v_0) | s \geq 0 \} \cap L_{v_{i/m}} \cap B_\epsilon (v_{i/m}).
    \end{equation*}
\end{Lem}
\begin{proof}
    Let $\delta = \dist{[v_0, v_1]}{\partial Y^*_{x_0}}$, and fix $\epsilon$ such that
    \begin{equation*}
        \epsilon < \min \left\{ \delta, \frac{1}{2} |v_1 - v_0| \right\}.
    \end{equation*}
    Note that $\epsilon < \frac{1}{2}r_k$. Let $m$ be a positive integer satisfying 
    \begin{equation}
    \label{eqn: choose m}
        \frac{|v_1 - v_0|}{\epsilon \sqrt{1-\frac{1}{k^2}}} < m < \frac{2|v_1 - v_0|}{\epsilon}.
    \end{equation}
    Observe that this readily implies \eqref{eqn: m epsilon}. Using the argument from Remark \ref{Remark : icone}, we obtain
    \begin{equation}
    \label{eqn: cone in SL}
        \begin{aligned}
            & SL_{v_{i/m}} \cap B_{\epsilon}(v_{i/m}) \\
            =& \left( \bigcap_{v'} \left\{ v | \langle v-v', \nabla F(v') \rangle \leq 0 \right\}\right) \cap B_\epsilon (v_{i/m}) \\
            \supset& \left( \bigcap_{v'} \left\{ v | \langle v-v_{i/m}, \nabla F(v') \rangle \leq 0 \right\}\right) \cap B_\epsilon (v_{i/m}) \\
            \supset& \left\{ v| \langle v-v_{i/m} , u \rangle \leq 0, \forall u \in B_{\frac{\rho_0}{2k}}(\nabla F(v_0))\right\} \cap B_\epsilon(v_{i/m}),
        \end{aligned}
    \end{equation}
    where $\displaystyle \bigcap_{v'}$ means the intersection over $v' \in L_{v_0}\cap B_{r_k}(v_{i/m})$, while we have used $B_\epsilon (v_{i/m}) \subset B_{r_k}(v_0)$. Up to an isometry, we can set 
    \begin{equation*}
        \nabla F(v_0) = \rho_0e_n, \quad v_0 = 0, \quad v_1 = ae_1 + be_n.
    \end{equation*}
    Then \eqref{eqn: cone in SL} implies
    \begin{equation*}
        \begin{aligned}
            \mathcal{D}_i & : = \left\{ v_{i/m} - \frac{\epsilon}{k}e_n + \epsilon r\hat{e} \Bigg| \langle \hat{e} , e_n \rangle = 0 , \| \hat{e} \| = 1,\ 0 \leq r \leq \sqrt{1-\frac{1}{k^2}}\right\} \\
            & \subset SL_{v_{i/m}}.
        \end{aligned}
    \end{equation*}
    Then, for any $t \in [\frac{i}{m}, \frac{i+1}{m}]$, by \eqref{eqn: choose m}, we have
    \begin{equation*}
        \begin{aligned}
            w_t & := v_t - \left( \left( t-\frac{i}{m}\right)b + \frac{\epsilon}{k} \right)e_n \\
            & = v_{i/m} - \frac{\epsilon}{k}e_n + \left(t-\frac{i}{m} \right)ae_1\\
            & \in \mathcal{D}_i \subset SL_{v_{i/m}}
        \end{aligned}
    \end{equation*}
    ($c$DomConv) implies 
    \begin{equation*}
        [v_t, w_t] \subset Y^*_{x_0}.
    \end{equation*}
    In addition, $w_t \in SL_{v_{i/m}}$ shows $F(w_t) \leq F(v_{i/m})$, while Remark \ref{Rmk : increasing in half ball} yields $F(v_t) \geq F(v_{i/m})$. Therefore, by the intermediate value theorem, we obtain $u_t \in [v_t, w_t]$ such that
    \begin{equation*}
        F(u_t) = F(v_{i/m}), \quad \textrm{i.e.} \quad u_t \in L_{v_{i/m}}. 
    \end{equation*}
    Note that the form of $w_t$ yields that $u_t \in \{v_t - s \nabla F(v_0)| s\geq 0 \}$. Finally, Noting that $\mathcal{D}_i \subset B_\epsilon(v_{i/m})$, we have $w_t \in B_\epsilon (v_{i/m})$. Also, $|v_t - v_{i/m}| \leq \frac{1}{m} |v_1 - v_0 | < \epsilon$, so that $v_t \in B_\epsilon(v_{i/m})$. Thus, convexity of $B_\epsilon(v_{i/m})$ yields $u_t \in B_\epsilon(v_{i/m})$. this concludes the proof of the lemma.
\end{proof}

\begin{Lem}
\label{Lem : Local QQconv 2}
    Let $v_0, v_1 \in \Int{Y^*_{x_0}}$ and suppose $v_1 \in \Int{B^+_{r_k/2}(v_0)}$. Then we have
    \begin{equation*}
        F(v_t) -F(v_0) \leq M_k t(F(v_1) - F(v_0))
    \end{equation*}
    for some constant $M_k$ that only depends on $k$.
\end{Lem}
\begin{proof}
    Let 
    \begin{equation*}
        u_{i/m} \in \left\{ v_{i/m}-s\nabla F(v_0) | s \geq 0 \right\} \cap L_{v_{(i-1)/m}} \cap B_\epsilon (v_{i/m}),
    \end{equation*}
    where $m$ and $\epsilon$ are from Lemma \ref{Lem : subdivision}. Then we compute
    \begin{equation*}
        \begin{aligned}
            F(v_{i/m}) - F(v_{(i-1)/m}) & = F(v_{i/m}) - F(u_{i/m}) \\
            & = \int_0^1 \langle \nabla F( s v_{i/m} + (1-s)u_{i/m}) , v_{i/m} - u_{i/m} \rangle  ds \\
            & = \int_0^1 \left\langle \nabla F( s v_{i/m} + (1-s)u_{i/m}), \frac{\nabla F(v_0)}{| \nabla F(v_0) |} \right\rangle | v_{i/m} - u_{i/m}| ds \\
            & \geq \left( 1 - \frac{1}{2k} \right) | \nabla F(v_0 ) | \times \frac{\epsilon}{2},
        \end{aligned}
    \end{equation*}
    where we have used \eqref{lower comparability} with $k'=1$ and \eqref{eqn: m epsilon} in the above inequality. Hence
    \begin{equation}
    \label{eqn: Fv1-Fv0 subdiv}
        \begin{aligned}
            F(v_1) - F(v_0) & = \sum_{i=1}^m (F(v_{i/m}) - F(v_{(i-1)/m})) \\
            & \geq \frac{1}{2} \left( 1 - \frac{1}{2k} \right) | \nabla F(v_0 ) | m\epsilon \\
            & \geq \frac{1}{2} \left( 1 - \frac{1}{2k} \right) | \nabla F(v_0 ) | |v_1 - v_0|.
        \end{aligned}
    \end{equation}
    Next, we fix $t \in [\frac{j-1}{m}, \frac{j}{m}]$. Then
    \begin{equation*}
        \begin{aligned}
            F(v_t) - F(v_0) & = (F(v_t) - F(v_{j/m})) + \sum_{i=1}^j (F(v_{i/m}) -F(v_{(i-1)/m}) ) \\
            & =: I_1 + I_2.
        \end{aligned}
    \end{equation*}
    We estimate $I_2$ first. 
    \begin{equation*}
        \begin{aligned}
            F(v_{i/m})- F(v_{(i-1)/m}) & = \int_0^1 \langle \nabla F (s v_{i/m} + (s-1)v_{(i-1)/m}), v_{i/m} - v_{(i-1)/m} \rangle ds \\
            & \leq \left( 1 + \frac{1}{2k} \right) | \nabla F(v_0)| \times \frac{|v_1 -v_0|}{m},
        \end{aligned}
    \end{equation*}
    where we have used \eqref{grad in small ball} in the above inequality. Hence, we obtain
    \begin{equation}
    \label{eqn: I2}
        \begin{aligned}
            I_2 & = \sum_{i=1}^j (F(v_{i/m}) - F(v_{(i-1)/m})) \\
            & \leq \sum_{i=1}^j \left( 1 + \frac{1}{2k} \right) | \nabla F(v_0)| \frac{|v_1 -v_0|}{m} \\
            & = \left( 1 + \frac{1}{2k} \right)\frac{j}{m} | \nabla F(v_0)|  |v_1 -v_0|
        \end{aligned}
    \end{equation}
    On the other hand, to estimate $I_1$, for $\tau \in [0, \frac{1}{m}]$, let
    \begin{equation*}
        u_{j/m + \tau} \in \{ v_{j/m+\tau} - s \nabla F(v_0) | s \geq 0\} \cap L_{v_{(j-1)/m}} \cap B_\epsilon (v_{(j-1)/m})
    \end{equation*}
    be from Lemma \ref{Lem : subdivision}. Then the convexity of the sublevel set $SL_{v_{(j-1)/m}}$ yields that the map
    \begin{equation*}
        \tau \mapsto | v_{j/m+\tau} - u_{i/m+\tau}| =: \sigma(\tau)
    \end{equation*}
    is convex on $(0,\frac{1}{m}]$. Then, defining $s_0 = \lim_{\tau \to 0}s_\tau = 0$, we obtain
    \begin{equation*}
        \sigma(\tau) \leq m\tau\sigma\left(\frac{1}{m}\right)+ (1-m\tau)\sigma(0) = m\tau \sigma\left(\frac{1}{m}\right).
    \end{equation*}
    Thus, using the above inequality with $\tau = t - j/m$, we obtain
    \begin{equation*}
        \begin{aligned}
            I_1 & = F(v_t) - F(u_t) \\
            & = \int_0^1 \langle \nabla F(sv_t + (1-s) u_t), v_t - u_t \rangle ds \\
            & = \int_0^1 \left\langle \nabla F(sv_t + (1-s) u_t) , \frac{\nabla F(v_0)}{|\nabla F(v_0)|}\right\rangle \sigma(\tau) ds \\
            & \leq \left( 1 + \frac{1}{2k} \right) | \nabla F(v_0) | \times m\tau \sigma(\frac{1}{m}) \\
            & = \left( 1 + \frac{1}{2k} \right) | \nabla F(v_0) | m \left( t - \frac{j}{m} \right) \left| v_{j/m} - u_{j/m}\right|.
        \end{aligned}
    \end{equation*}
    Noting that $v_{j/m}, u_{j/m} \in B_\epsilon(v_{(j-1)/m})$, we have 
    \begin{equation*}
        \left| v_{j/m} - u_{j/m}\right| \leq 2\epsilon \leq \frac{4}{m}|v_1 - v_0|,
    \end{equation*}
    and therefore
    \begin{equation*}
        \begin{aligned}
            I_1 \leq 4\left( 1 + \frac{1}{2k} \right) \left( t - \frac{j}{m} \right) | \nabla F(v_0) | |v_1 - v_0|.
        \end{aligned}
    \end{equation*}
    Combining the above estimate with \eqref{eqn: I2}, we obtain
    \begin{equation*}
        \begin{aligned}
            F(v_t) - F(v_0) & = I_1 + I_2 \\
            & \leq  4\left( 1 + \frac{1}{2k} \right) t | \nabla F(v_0) | |v_1 - v_0|.
        \end{aligned}
    \end{equation*}
    Using \eqref{eqn: Fv1-Fv0 subdiv}, we end up with
    \begin{equation*}
        \begin{aligned}
            F(v_t) - F(v_0) & \leq 8\left( 1 + \frac{1}{2k} \right)\left( 1-\frac{1}{2k} \right)^{-1}(F(v_1) - F(v_0)) \\
            &= \frac{16k+8}{2k-1} t (F(v_1) - F(v_0)).
        \end{aligned}
    \end{equation*}
    This concludes the proof of the lemma.
\end{proof}

Now we give the proof of the main theorem.

\begin{proof}[Proof of the main theorem Theorem \ref{MainThm}]
Fix  a positive integer $k$. Lemma \ref{Lem : Local QQconv 2} yield that for any $v_0 \in \Int{Y^*_{x_0}}$ and for any $v_1 \in \Int{B^+_{r_k/2}(v_0)}$, we have
\begin{equation}
\label{eqn: QQconv Mk}
    F(v_t) -F(v_0) \leq M_k t(F(v_1) - F(v_0))
\end{equation}
where $M_k$ only depends on $k$. Then, for any point $v_0\in Y^*_{x_0}$ and $v_1 \in B^+_{r_k/2}(v_0)$, we can choose sequences $v^i_0 \in \Int{Y^*_{x_0}}$ and $v_1^i \in \Int{B^+_{r_k/2}(v_0^i)}$ that each sequence converges to $v_0$ and $v_1$ respectively as $i \to \infty$. Since, for each $(v_0^i,v_1^i)$, we have
\begin{equation*}
    F(v_t^i) -F(v_0^i) \leq M_k t(F(v_1^i) - F(v_0^i)),
\end{equation*}
where $v_t^i = tv_1^i + (1-t) v_0^i$. Taking limit $i \to \infty$, we obtain \eqref{eqn: QQconv Mk}. Then, by Lemma \ref{Lem : QQconv on half ball to global}, we conclude that $c$ is QQconv.
\end{proof}

\section{A remark on a counter example}
\label{counter ex}
To state the Loeper's condition and QQconv properly, we need at least $C^1$ regularity with the (Twisted) and (Twisted*) conditions due to $\cexp$ in the inequalities. In addition, (cDomConv) is needed to define the $c$-segment $[v_0, v_1]$. Moreover, without the $C^2$ regularity in the sense of the existence of the mixed Hessian $D^2_{xy}c = (D^2_{yx}c)^T$, we won't be able to define the (Non-degenerate) condition. Thus (Non-degenerate) condition is the only condition that can be removed from Theorem \ref{MainThm} without changing the other parts of the statement of the theorem. In this case, however, the theorem becomes false. i.e. there is a cost function $c$ that satisfies Loeper's condition, but not QQconv.

If we assume that (Non-degenerate) condition does not hold at some point, then, heuristically, from \eqref{grad F}, we can expect that $|\nabla F| = \infty$ at the point. Then, from \eqref{QQconvex ineq 2}, formally, we may have
\begin{equation*}
    F(v_1) - F(v_0) \geq \lim_{t \to \infty}\frac{1}{t}(F(v_t) - F(v_0)) = \langle \nabla F(v_0), v_1 - v_0 \rangle = \infty
\end{equation*}
for some $v_1$, which yields that QQconv cannot hold.

To build a concrete example, we use the cost function of the form
\begin{equation*}
    c(x,y) = \frac{1}{2}|x-y|^2 +\frac{1}{2}(f(x)-g(y))^2,
\end{equation*}
where $f$ and $g$ are convex functions. This cost function is known to satisfy $(A3w)$ condition when $f$ and $g$ are $C^2$ and $\langle \nabla f, \nabla g \rangle > -1$ (see \cite{MTW}). Hence, the cost functions of the above form also satisfies the Loeper's condition. We let 
\begin{equation*}
    f(x) = (a_1)^2, \quad g(y) = (b_1)^2
\end{equation*}
where $x=(a_1, a_2)$ and $y=(b_1, b_2)$ are the points in $\R^2$. Then we have
\begin{equation}
\label{eqn: counter ex Dc}
    \begin{aligned}
        &-D_x c(x, y) = y-x + 2((b_1)^2-(a_1)^2)a_1e_1, \\
        &-D_y c(x, y) = x-y + 2((a_1)^2-(b_1)^2)b_1e_1,
    \end{aligned}
\end{equation}
which are injective in $b_1\leq \frac{-1}{4a_1}$ for fixed $x$ and $a_1 \geq \frac{-1}{4b_1}$ for fixed $y$ respectively. Note also that they are injective when $a_1=0$ and $b_1=0$ respectively. We let 
\begin{equation*}
    X = \left[ -\frac{1}{2}, 0 \right]\times [0,1], \quad Y = \left[ 0, \frac{1}{2} \right]\times [0,1].
\end{equation*}
Then the cost function $c:X\times Y \to \R$ satisfies (Twisted) and (Twisted*). Also, we can compute that, for $x\in X$ and $y \in Y$,
\begin{equation*}
    \begin{aligned}
        & Y_x^* = \left[ -a_1 -2(a_1)^3, \frac{1}{2}-\frac{1}{2}a_1 -2(a_1)^3 \right] \times \left[ -a_2, 1-a_2 \right], \\
        & X_y^* = \left[-\frac{1}{2}-\frac{1}{2}b_1 - 2(b_1)^3, -b_1-2(b_1)^3  \right] \times \left[ -b_2, 1 -b_2\right].
    \end{aligned}
\end{equation*}
Hence, $X$ and $Y$ satisfy ($c$DomConv) and ($c$DomConv*) (See \cite{MTW} for further discussion about the convexity of the domains). Moreover, for $(x,y) \in X \times Y$ and $ v = (\gamma_1, \gamma_2)  \in Y^*_x$, \eqref{eqn: counter ex Dc} implies that the $c$-exponential map is 
\begin{equation}
\label{eqn: counter ex cexp}
    \begin{aligned}
        \cexp(v) & = \left( \left( \frac{\gamma_1}{2a_1} + \left( a_1 + \frac{1}{4a_1}\right)^2 \right)^{\frac{1}{2}} - \frac{1}{4a_1}, \gamma_2+a_2 \right).
    \end{aligned}
\end{equation}
Since the cost function $c$ satisfies $(A3w)$ on the region where $\langle \nabla f, \nabla g \rangle > -1$, we have 
\begin{equation}
\label{eqn: counter ex Loeper}
    -c(x_1,y_t) +c(x_0, y_t) \leq \max\{-c(x_1,y_1) +c(x_0, y_1), -c(x_1,y_0) +c(x_0, y_0) \},
\end{equation}
where $x_i=(a_1^i, a_2^i) \in X$ and $y_i = \cexpz(v_i)$ for some $v_i$ such that
\begin{equation*}
    v_i \in \left( -a_1^0 -2(a_1^0)^3, \frac{1}{2}-\frac{1}{2}a_1^0 -2(a_1^0)^3 \right] \times \left[ -a_2^0, 1-a_2^0 \right]'
\end{equation*}
and $y_t= \cexpz(v_t) = \cexpz(tv_1 + (1-t)v_0)$. Noting that \eqref{eqn: counter ex cexp} is stable under the limit $\gamma_1 \to \frac{1}{2}-\frac{1}{2}a_1^0 -2(a_1^0)^3$, we can take the limit in \eqref{eqn: counter ex Loeper} and conclude that the same inequality holds even if we choose $v_i \in Y^*_{x_0}$ that has the first coordinate $\frac{1}{2}-\frac{1}{2}a_1^0 -2(a_1^0)^3$. i.e. $y_i$ has the first coordinate $\frac{1}{2}$. Observe that, if we choose $x_0 = -\frac{1}{2}e_1$ and $y_0=\frac{1}{2}e_1$, where $e_1 = (1,0)$, the previous argument asserts that \eqref{eqn: counter ex Loeper} holds while $\langle \nabla f(x_0), \nabla g(y_0) \rangle = -1$. In particular, Loeper's property holds at the points where the cost function does not satisfy (Non-degenerate) condition. Note also that the MTW tensor is not well-defined at this point. \\
We choose
\begin{equation*}
    \begin{aligned}
        x_0 &= - \frac{1}{2} e_1 ,  &x_1 = 0, \\
        y_0 &= \frac{1}{2} e_1 ,  &y_1 = 0,
    \end{aligned}
\end{equation*}
and let $v_i \in Y^*_{x_0}$ be such that $\cexpz(v_i) = y_i$. Then, the function $F$ defined in \eqref{Def : F} is
\begin{equation*}
    \begin{aligned}
        F(v) &= \langle \frac{1}{2}e_1, \cexpz(v) \rangle + \frac{1}{4} (b_1)^2 + R \\
        & = \frac{1}{2}b_1+ \frac{1}{4} (b_1)^2 + R 
    \end{aligned}
\end{equation*}
for some constant $R$ where $b_1$ is the first coordinate of $\cexpz(v)$. i.e., for $v=(\gamma_1, \gamma_2)$,
\begin{equation*}
    b_1 = (1-\gamma_1)^{\frac{1}{2}}+\frac{1}{2}.
\end{equation*}
Also, we have
\begin{equation*}
    \begin{aligned}
        v_0 = e_1, \quad v_1 = \frac{3}{4}e_1.
    \end{aligned}
\end{equation*}
Then we compute
\begin{equation*}
    F(v_1) = \frac{3}{4} +R , \quad F(v_0) = \frac{5}{16}+R.
\end{equation*}
In particular, $F(v_1) - F(v_0) >0$. On the other hand, using the mean value theorem, we obtain
\begin{equation*}
    \begin{aligned}
        \lim_{t \to 0} \frac{F(v_t)-F(v_0)}{t} & = \lim_{t \to 0} \langle \nabla F(v_{s(t)}), -\frac{1}{4}e_1\rangle \\
        & = \lim_{t \to 0 } \left( \frac{1}{8s(t)^{\frac{1}{2}}} + \frac{s(t)^{\frac{1}{2}}+1}{16s(t)^{\frac{1}{2}}} \right) \\
        & \to \infty.
    \end{aligned}
\end{equation*}
Thus, for any fixed $M$, there exists $t \in (0,1)$ such that 
\begin{equation*}
    F(v_t)-F(v_0) > M(F(v_1) - F(v_0)).
\end{equation*}
This yields that the cost function $c$ does not satisfy QQconv.

\bibliographystyle{plain}

\bibliography{ref}

\end{document}